\documentclass[11pt]{amsart}
\usepackage{amssymb, latexsym}
\theoremstyle{plain}
\newtheorem{theorem}{Theorem}
\newtheorem{corollary}{Corollary}
\newtheorem*{thm*}{Theorem}
\newtheorem {lemma}{Lemma}
\newtheorem{proposition}{Proposition}

\numberwithin{equation}{section}

\begin{document}

\title[ A secretary for Messrs Luce and Mallows ] {A secretary for Messrs. Luce and Mallows }

\author{Ross G. Pinsky}

%\noindent  pinsky@math.technion.ac.il\ \ \ \ tel: 972-4-829-4083\ \ \  fax: 972-4-829-3388

\address{Department of Mathematics\\
Technion---Israel Institute of Technology\\
Haifa, 32000\\ Israel}
\email{ pinsky@technion.ac.il}

\urladdr{https://pinsky.net.technion.ac.il/}

\date{}

\begin{abstract}

We analyze the secretary problem in the case that the $n$ ranked items arrive
not in uniformly random  order but rather according to a  certain type of Luce distribution or according to a Mallows distribution on the set $S_n$ of permutations of $[n]$.
The secretary problem for the Mallows distributions with parameter $q\in(0,1)$ was analyzed in a previous paper; in this paper the case $q>1$ is also analyzed.
The Luce distribution with the  class $\{q_j\}_{j=1}^\infty$ of weights is related in a certain sense to  the Mallows distribution with parameter $q$, but is more difficult to analyze.
It turns out that for every $n$ and every strategy, the probabilities for the  secretary problem when the smallest number is considered of highest rank for the Luce distribution with this  class of weights coincides with those for the corresponding Mallows distribution.
We  analyze the asymptotic optimal strategy and corresponding limiting probability for the above cases, as well as for the Luce distributions with other classes of
weights, such as the Sukhatme weights.
\end{abstract}

\maketitle
\section{Introduction and Statement of Results}\label{intro}
Recall the classical secretary problem: For $n\in\mathbb{N}$, a set of $n$ ranked items is revealed, one item at a time, to an observer whose objective is to select the item with
the highest rank. The order of the items is completely random; that is, each of the  $n!$ permutations of the ranks is equally likely.
At each stage, the observer only knows the relative ranks of the items that have arrived thus far, and must  either select the current item, in which case the process terminates, or reject it and continue to the next item. If the observer rejects the first $n-1$ items, then the $n$th and final item to arrive must be accepted.
As is very well known, asymptotically as $n\to\infty$, the optimal  strategy is to reject the first $M_n$ items, where $M_n\sim \frac ne$, and then to select the first
later-arriving item whose rank is higher than that of any of the first $M_n$ items (if such an item exists).
The limiting probability of successfully selecting the item of highest rank is $\frac1e$.

One can consider the secretary problem when the permutation indicating the relative ranks of the items is not uniformly distributed, but rather  distributed according to
some other probability measure on the set $S_n$ of permutations of $[n]$.
When the distribution is not the uniformly  random one, the case that the largest  number, that is  $n$,  is considered  to be of  highest rank and the case that the smallest number, that is  1, is considered to be of  highest rank must be analyzed separately.
In \cite{P22a}, we initiated a study of the secretary problem when the permutation indicating the relative ranks of the items is distributed according to a Mallows distribution with parameter $q=q_n\in(0,1)$ and the largest number is of highest rank.
In this paper, we will complete the analysis of the secretary problem for a Mallows distribution by also allowing  $q=q_n>1$, and by also considering the case that the smallest number is of highest rank.
The Mallows distributions are described two paragraphs  below. The major motivation for completing this analysis comes from the main focus of this paper, which is the analysis of the  secretary problem when a Luce distribution is placed on $S_n$.
The Luce distributions on $S_n$, which we now describe,  come in two types and  are parameterized by a set of positive weights $\{\theta_i\}_{i=1}^n$.
It turns out that there is an intimate connection between the secretary problem for the Mallows distribution and for the secretary  problem
for one of the two types of Luce  distributions for appropriate choices of the weights, in the case that the smallest number is  of highest rank.

For $n\in\mathbb{N}$ and a collection $\{\theta_i\}_{i=1}^n$ of positive weights,  the first type of  Luce distribution $P_n^{\mathcal{L};\{\theta_i\}_{i=1}^n}$ is defined by
\begin{equation}\label{Luce}
P_n^{\mathcal{L};\{\theta_i\}_{i=1}^n}(\sigma_1\cdots\sigma_n)=\theta_{\sigma_1}\frac{\theta_{\sigma_2}}{\omega_n-\theta_{\sigma_1}}\frac{\theta_{\sigma_3}}{\omega_n-\theta_{\sigma_1}-\theta_{\sigma_2}}
\cdots\frac{\theta_{\sigma_n}}{\theta_{\sigma_n}},\ \text{where}\ \omega_n=\sum_{i=1}^n\theta_i.
\end{equation}
Here we have written the permutation $\sigma\in S_n$ in one line notation, $\sigma=\sigma_1\cdots\sigma_n$.
(Note that $\{\theta_i\}_{i=1}^n$ and $\{c\theta_i\}_{i=1}^n$, for $c>0$, yield the same distribution, so one can always normalize and assume that $\omega_n=1$.)
The second type of Luce distribution, which we denote by $P_n^{\mathcal{L}_\text{inv};\{\theta_i\}_{i=1}^n}$ is just the measure induced by the first type of Luce distribution via the map that takes a permutation $\sigma\in S_n$ to its inverse $\sigma^{-1}$.
Thus,
\begin{equation}\label{Luceinv}
\begin{aligned}
&P_n^{\mathcal{L}_\text{inv};\{\theta_i\}_{i=1}^n}(\sigma_1\cdots\sigma_n)=\theta_{\sigma^{-1}_1}\frac{\theta_{\sigma^{-1}_2}}{\omega_n-\theta_{\sigma^{-1}_1}}
\frac{\theta_{\sigma^{-1}_3}}{\omega_n-\theta_{\sigma^{-1}_1}-\theta_{\sigma^{-1}_2}}
\cdots\frac{\theta_{\sigma^{-1}_n}}{\theta_{\sigma^{-1}_n}},\\
& \text{where}\ \omega_n=\sum_{i=1}^n\theta_i.
\end{aligned}
\end{equation}

The Luce distributions have an online construction
via weighted sampling without replacement. Consider a box containing balls numbered from 1 to $n$. For each $i\in[n]$, at stage $i$ a ball is picked with probability proportional to its weight relative to the weights of all the remaining balls. The number of the ball picked at stage $i$ is the value of $\sigma_i$ under
$P_n^{\mathcal{L};\{\theta_i\}_{i=1}^n}$  and is the value of $\sigma^{-1}_i$ under $P_n^{\mathcal{L}_\text{inv};\{\theta_i\}_{i=1}^n}$.
The Luce distributions arise in many natural ways, one  of which will be seen further on in this paper; see \cite{CDK} for  examples and references, and see also
\cite{BCD}.
We note that \cite{CDK} considers the Luce measure
$P_n^{\mathcal{L};\{\theta_i\}_{i=1}^n}$ while \cite{BCD} studies the Luce measure $P_n^{\mathcal{L}_\text{inv};\{\theta_i\}_{i=1}^n}$.

The Mallows distributions are obtained  by exponentially tilting via the inversion statistic.
Recall that the inversion statistic $I_n$ is defined
by
$$
I_n(\sigma)=\sum_{1\le i<j\le n} 1_{\{\sigma_j<\sigma_i\}}=\sum_{1\le i<j\le n}1_{\{\sigma^{-1}_i<\sigma^{-1}_j\}},\ \sigma\in S_n.
$$
The Mallows distribution on $S_n$ with parameter $q\in(0,\infty)$, which we denote by $P_n^{\mathcal{M};q}$, is defined
by
\begin{equation}\label{Mallows}
P_n^{\mathcal{M};q}(\sigma)=\frac{q^{I_n(\sigma)}}{Z_n(q)},\ \sigma\in S_n,
\end{equation}
where the normalization constant $Z_n(q)$ is given by
\begin{equation}\label{Zn}
Z_n(q)=\frac{\prod_{j=1}^n(1-q^j)}{(1-q)^n},\ \text{for}\ q\neq1.
\end{equation}
When $q=1$, $Z_n(1)=n!$ and the distribution is the uniform one.
Note that for $q\in(0,1)$, permutations with fewer inversions have higher probabilities than those with more inversions, and vice versa when $q>1$.
Note also that since $I_n(\sigma)=I_n(\sigma^{-1})$, the  Mallows distributions are invariant under the map $\sigma\to\sigma^{-1}$.
There are two alternative online methods for creating  a Mallows permutation; they   will be described below.

Because of the nature of  its online construction described above, a Luce distribution of the type  $P_n^{\mathcal{L};\{\theta_i\}_{i=1}^n}$ has sometimes been called in the literature a $p$-biased distribution; see \cite{PT} and references therein, and  \cite{CRSBDJN}. (The notation $\{p_i\}$ is used there instead of $\{\theta_i\}$.)
Here is another  online construction using the $\{\theta_i\}_{i=1}^n$; it leads to what is called a  $p$-shifted distribution; see \cite{PT} and references there in, and
\cite{GO}.
Again consider the box with the balls numbered from 1 to $n$.
At stage one,  choose a ball, whose number will be the value for $\sigma_1$, with probabilities proportional to the weights $\{\theta_i\}_{i=1}^n$.
If $\sigma_1=l$, then at stage two, choose one of the remaining balls with probabilities proportional to the weights $\{\theta_i\}_{i=1}^{n-1}$, where balls number 1 though $l-1$ respectively receive weights $\theta_1,\cdots,\theta_{l-1}$, and balls number $l+1$  though $n$ respectively receive weights $\theta_l,\cdots, \theta_{n-1}$.
One continues like this, thus  at stage $m+1$, if one has $\sigma_1\cdots\sigma_m=l_1\cdots l_m$, then
the remaining $n-m$ numbers, $[n]-\{l_r\}_{r=1}^m$, arranged in increasing order, are given respectively the weights $\{\theta_i\}_{i=1}^{n-m}$.
Here is an example to illustrate the construction. Let $n=4$, and let $\theta_i=i$. Then $\sigma_1$ is chosen from  $\{1,2,3,4\}$
with respective     probabilities $\frac1{10},\frac2{10},\frac3{10},\frac4{10}$. Now say that $\sigma_1=3$. Then $\sigma_2$ is chosen from
$\{1,2,4\}$ with respective probabilities $\frac16,\frac26,\frac36$. Now say that $\sigma_2=1$. Then  $\sigma_3$ is chosen from $\{2,4\}$ with respective probabilities
$\frac13,\frac23$. Then $\sigma_4$ is equal to the final remaining number.
It turns out that the Mallows distribution with parameter $q$ is equal to the above $p$-shifted distribution in the case that $\theta_i=q^i$ \cite{GO, P22b}.
Note that as with the $p$-biased construction leading to the Luce distributions of the two types, we could also create two types of distributions
in the case of $p$-shifted by also looking at the measure induced by the map $\sigma\to\sigma^{-1}$. However, in this paper we are only interested in the
$p$-shifted distribution as it relates to the Mallows distributions, and as we have already noted, the Mallows distributions are invariant under the map
 $\sigma\to\sigma^{-1}$.
Some comparisons between the Luce distribution with weights $\{\theta_i\}_{i=1}^n=\{q^i\}_{i=1}^n$ and the Mallows distribution with parameter $q$ are noted in the remark following Corollary \ref{LM} below.

Although the description of the construction of the $p$-shifted distribution is more complicated than that of the $p$-biased (Luce) distributions, the $p$-shifted distributions are much more tractable than the Luce distributions.
It is noted in \cite{CDK} that under a Luce distribution,  the behavior of most of the standard permutation statistics is still open.
The reason that   the $p$-shifted distribution is quite tractable is
 that the Lehmer code for this distribution is a set of independent random variables---see Proposition 1.7 in \cite{P22b}.
What this means is that the $p$-shifted distribution with weights $\{\theta_i\}_{i=1}^n$   has the following alternative online construction.
Let $\{X_j\}_{j=2}^n$ be independent random variables with distributions given by
\begin{equation}\label{X-inv}
P(X_j=i-1)=\frac{\theta_i}{\sum_{l=1}^j\theta_l}, \ i=1,\cdots, j.
\end{equation}
At stage one, place the number 1 down on an empty row. For each $j=2,\cdots n$, at stage $j$ when the numbers $1\cdots j-1$ are already on the row, place the number $j$ down so that $X_j$ numbers are to its right. It can be shown that
the permutation $\sigma$ thus constructed has the $p$-shifted distribution with weights $\{\theta_i\}_{i=1}^n$.
Note that the number of inversions in $\sigma$ is then given by $I_n(\sigma)= \sum_{j=2}^nX_j$.
The analysis of the secretary problem for the Mallows distribution with parameter $q$ relies on this construction with $\theta_i=q^i$, and also on the original definition via exponential tilting in
\eqref{Mallows}.

We now turn to the analysis of the secretary problem.
For the secretary problem where the largest number is  of highest rank,
denote by $\mathcal{S}^{\uparrow}(n,M)\subset S_n$ the event that the highest ranked item $n$ is selected via  the strategy that rejects the first $M$ items  and then selects the first later-arriving item  whose rank is higher than any of the first $M$ items (if such an item exists).
Similarly, for the secretary problem where the smallest number is  of highest rank,
denote by $\mathcal{S}^{\downarrow}(n,M)\subset S_n$ the event that the highest ranked item $1$ is selected via  the strategy that rejects the first $M$ items  and then selects the first later-arriving item  whose rank is higher than any of the first $M$ items (if such an item exists).

We begin with exact results for every $n$, and then use them to obtain more interesting asymptotic results as $n\to\infty$.
In \cite{P22a} we proved the following theorem concerning the secretary problem for the Mallows distribution
where the largest number is  of highest rank, and with
parameter $q\in(0,1)$. It is easy to check that the same formula holds  and the same proof works also
for $q>1$.

\begin{thm*} {\bf(P-1)}
Consider the secretary problem for a Mallows distribution in the case that the largest number is of highest rank.
For  $n\in\mathbb{N}$  and $q\in(0,\infty)-\{1\}$,
\begin{equation}\label{exactform}
P_n^{\mathcal{M};q}(\mathcal{S}^\uparrow(n,M))=\begin{cases}\frac{1-q}{1-q^n}q^{n-M-1}(1-q^M)\sum_{j=M+1}^n\frac1{1-q^{j-1}},\  M\in\{1,\cdots, n-1\};\\
\frac{1-q}{1-q^n}q^{n-1},\ M=0.\end{cases}
\end{equation}
\end{thm*}
The secretary problem for the Mallows distribution when the smallest number is of highest rank is  obtained easily as a corollary of Theorem P-1.
\begin{corollary}\label{secminmax}
The secretary problem for Mallows distributions when the smallest number is  of highest rank and the secretary problem for the Mallows distributions when the largest number is  of highest rank are related by the following formula:
\begin{equation}\label{sec-rankminmax}
P_n^{\mathcal{M};q}(\mathcal{S}^\downarrow(n,M))=P_n^{\mathcal{M};\frac1q}(\mathcal{S}^\uparrow(n,M)), \ q\in(0,\infty).
\end{equation}
Thus from \eqref{exactform}, for $q\in(0,\infty)-\{0\}$,
\begin{equation}\label{exactformmin}
P_n^{\mathcal{M};q}(\mathcal{S}^\downarrow(n,M))=\begin{cases}\frac{1-q}{1-q^n}(1-q^M)\sum_{j=M+1}^n\frac{q^{j-1}}{1-q^{j-1}},\ \text{if}\ M\ge1;\\
\frac{1-q}{1-q^n},\ \text{if}\ M=0.\end{cases}
\end{equation}

\end{corollary}

\begin{proof}
Recall that the complement $\sigma^{\text{comp}}$ of a permutation $\sigma=\sigma_1\cdots\sigma_n$ is defined by $\sigma^{\text{comp}}_i=n+1-\sigma_i,\ i\in[n]$.
The number of inversions in $\sigma$ and in $\sigma^{\text{comp}}$ are related by
$I_n(\sigma^{\text{comp}})=\frac {n(n-1)}2-I_n(\sigma)$, From \eqref{Zn}, it follows that $Z_n(q)=q^{\frac{n(n-1)}2}Z_n(\frac1q)$. Thus, from \eqref{Mallows},
$$
P_n^{\mathcal{M};q}(\sigma^{\text{comp}})=\frac{q^{\frac{n(n-1)}2-I_n(\sigma)}}{Z_n(q)}=\frac{(\frac1q)^{I_n(\sigma)}}{Z_n(\frac1q)}=P_n^{\mathcal{M};\frac1q}(\sigma),
$$
which shows that the distribution induced by the Mallows distribution with parameter $q$ under the map $\sigma\to\sigma^{\text{comp}}$ is the Mallows distribution with parameter $\frac1q$.
    It is clear that the  secretary  problem for a permutation $\sigma$ where the
largest number is  of highest rank is identical to the secretary  problem for $\sigma^\text{comp}$  where the
 smallest number is  of highest rank.
 Equation \eqref{exactformmin} follows by replacing $q$ by $\frac1q$ in \eqref{exactform}.
\end{proof}
\noindent \bf Remark.\rm\ This is the appropriate place to correct an incorrect comment that we wrote in \cite{P22a}.
We claimed there that  it suffices to consider the case $q\in(0,1)$ because
if $\sigma$ is distributed according to a Mallows distribution with parameter $q$, then
the reverse $\sigma^\text{rev}=\sigma_n\cdots\sigma_1$ is distributed according to a Mallows distribution with parameter $\frac1q$.
But the secretary problem for the reverse of a permutation is not the same as the secretary problem for the original permutation.

We now turn to the secretary problem for the Luce distributions  $P_n^{\mathcal{L}_\text{inv};\{\theta_i\}_{i=1}^n}$ when the smallest number is of highest rank.
This is the only situation that we can handle for the Luce distributions. That is, for $P_n^{\mathcal{L}_\text{inv};\{\theta_i\}_{i=1}^n}$ , we cannot handle the case  that the largest number is of highest rank, and for $P_n^{\mathcal{L};\{\theta_i\}_{i=1}^n}$ we cannot handle either case. As will be seen in the proof of the following theorem, what allows us to handle the secretary problem for the $P_n^{\mathcal{L}_\text{inv};\{\theta_i\}_{i=1}^n}$-type Luce distributions when the smallest number is of highest rank is
the representation of these distributions via \eqref{Red} in Theorem CDK below, along with the fact that the minimum (but not the maximum) of independent exponential distributions has an exponential distribution.
Furthermore, the remark after Proposition \ref{BrussresultLuce} below suggests that these three other secretary problems are fundamentally different and harder than the one we analyze here.
\begin{theorem}\label{lucesec}
Consider the secretary problem for a
$P_n^{\mathcal{L}_\text{inv};\{\theta_j\}_{j=1}^n}$-type Luce
  distribution in the case that the smallest number is of highest rank.
For  $n\in\mathbb{N}$,
\begin{equation}\label{lucesecmin}
P_n^{\mathcal{L}_\text{inv};\{\theta_j\}_{j=1}^n}(\mathcal{S}^\downarrow(n,M))=
\begin{cases}\frac{\sum_{i=1}^M\theta_i}{\sum_{i=1}^n\theta_i}\sum_{j=M+1}^n
\frac{\theta_j}{\sum_{i=1}^{j-1}\theta_i},\ M\in\{1,\cdots, n-1\};\\
\frac{\theta_1}{\sum_{i=1}^n\theta_i},\ M=0.\end{cases}
%\frac{\theta_j}{\sum_{i=1}^n\theta_i}\thinspace\frac{\sum_{i=1}^M\thetaפ_i}{\sum_{i=1}^{j-1}\theta_i}.
\end{equation}
\end{theorem}
We have the following corollary.
\begin{corollary}\label{LM}
When the weights for the Luce distribution are given by $\theta_j=q^j$, for $q\in(0,\infty)$, one has
\begin{equation}\label{M=L}
P_n^{\mathcal{L}_\text{inv};\{q^j\}_{j=1}^n}(\mathcal{S}^\downarrow(n,M))=P_n^{\mathcal{M};q}(\mathcal{S}^\downarrow(n,M)).
\end{equation}
\end{corollary}
\begin{proof}
Substituting $\theta_j=q^j$ in  \eqref{lucesecmin}, a direct calculation reveals that the righthand side of \eqref{lucesecmin} is equal to the righthand side of
\eqref{exactformmin}.
\end{proof}
\bf\noindent Remark 1.\rm\
The weights $\theta_j=q^j,\ q\in(0,\infty)$, are called \it exponential Sukhatme\rm\ weights in \cite{BCD}.
(Actually, in \cite{BCD} one finds $\theta^{(n)}_j=q^{n+1-j}, \ i\in[n]$,
but recalling the definitions of the  Luce distributions, a moment's thought shows that these weights  are equivalent to the weights $\theta_i=(\frac1q)^i,\ i\in[n]$.)
 The \it Sukhatme weights\rm\ for the Luce distribution are defined below in \eqref{Sukh}.

\bf\noindent Remark 2.\rm\ It seems surprising  that the exact same formula for success holds for the secretary problem with  these two different families of measures, even though the position $\sigma^{-1}_1$ of the number 1 does have the same distribution under
$P_n^{\mathcal{M};q}$ and under $P_n^{\mathcal{L}_\text{inv};\{q^j\}_{j=1}^n}$.
Indeed,  from the definition of the Luce distribution $P_n^{\mathcal{L};\{q^j\}_{j=1}^n}$ in \eqref{Luce} or from the online description, under this distribution
the number $\sigma_1$  in the first position of the permutation has distribution
$P_n^{\mathcal{L};\{q^j\}_{j=1}^n}(\sigma_1=j)=\frac{q^j}{\sum_{i=1}^nq^i},\ j\in[n]$. It then follows from the definition of
$P_n^{\mathcal{L}_\text{inv};\{q^j\}_{j=1}^n}$ that the position $\sigma^{-1}_1$ of the number 1
has this same distribution under $P_n^{\mathcal{L}_\text{inv};\{q^j\}_{j=1}^n}$.
From the online ($p$-shifted) construction of the Mallows distribution $P_n^{\mathcal{M};q}$, it follows that under this distribution  the number $\sigma_1$ in the first position of the permutation
has this same distribution: $P_n^{\mathcal{M};q}(\sigma_1=j)= \frac{q^j}{\sum_{i=1}^nq^i},\ j\in[n]$.
As already noted, from the original definition of the Mallows distributions in \eqref{Mallows}, the Mallows distribution is invariant under the map
$\sigma\to\sigma^{-1}$. Thus the position $\sigma^{-1}_1$ of the number 1
has this same distribution under $P_n^{\mathcal{M};q}$.

For $q\in(0,1)$, from the point of view  of inversions, a random permutation $\sigma$ has more of a tendency
to be in increasing order
under the distribution $P_n^{\mathcal{L}_\text{inv};\{q^j\}_{j=1}^n}$
than under the distribution $P_n^{\mathcal{M};q}$. In particular, for example,
the total number of inversions statistic $I_n$ satisfies
$$
\begin{aligned}
&\text{w}-\lim_{n\to\infty}\frac{I_n}n=\sum_{k=1}^\infty\frac{q^k}{1+q^k},\ \text{under}\ P_n^{\mathcal{L}_\text{inv};\{q^j\}_{j=1}^n};\\
&\text{w}-\lim_{n\to\infty}\frac{I_n}n=\frac q{1-q},\ \text{under}\ P_n^{\mathcal{M};q},
\end{aligned}
$$
and $\frac q{1-q}>\sum_{k=1}^\infty\frac{q^k}{1+q^k}$.
See \cite[Proposition 1.2]{P22b}. (The Luce distribution $P_n^{\mathcal{L};\{q^j\}_{j=1}^n}$ is used in \cite{P22b}, but the inversion statistic has the same distribution
under the two types of Luce distributions since the inversion statistic is invariant under the map $\sigma\to\sigma^{-1}$.)
A similar analysis can be made when  $q>1$.

 A  classical result in optimal stopping due to  Bruss \cite{B} goes as follows. Let    $\{U_j\}_{j=1}^\infty$ be independent Bernoulli random variables taking values 0 and 1. Fix $n\in\mathbb{N}$ and observe the
$\{U_j\}_{j=1}^n$ sequentially. The object is to predict when the last 1 will appear and to stop at that stage.
Bruss's result is that the optimal strategy always yields a limiting probability of success  as $n\to\infty$ of at least $\frac1e$.
The secretary problems we consider here fit into this setup. In the case that the largest number is  of highest rank, let $U_j=1_{\{\sigma_j=\max(\sigma_1,\cdots, \sigma_j)\}}$ and in the case that
the smallest number is of highest rank, let  $U_j=1_{\{\sigma_j=\min(\sigma_1\cdots \sigma_j)\}}$.
We will  show that in the case of a Mallows distributions $P_n^{\mathcal{M};q}$,  with the
$\{U_j\}_{j=1}^n$  defined either with the maximum or with the minimum,
and in the case of the Luce distributions $P_n^{\mathcal{L}_\text{inv};\{\theta_j\}_{j=1}^n}$, with
the
$\{U_j\}_{j=1}^n$ defined with the minimum, the $\{U_j\}_{j=1}^n$ are independent.
This yields the following propositions.
\begin{proposition}\label{BrussresultLuce}
For each $n\in\mathbb{N}$, consider the secretary problem for the
 Luce distribution $P_n^{\mathcal{L}_\text{inv};\{\theta_i\}_{i=1}^n}$ when the smallest number is  of highest rank.
Denote by $M_n^*$ an optimal strategy.
Then
\begin{equation}\label{BrussLuce}
\liminf_{n\to\infty}P_n^{\mathcal{L}_\text{inv};\{\theta_i\}_{i=1}^n\}}(\mathcal{S}^\downarrow(n,M_n^*))\ge\frac1e.
\end{equation}

\end{proposition}

\bf\noindent Remark.\rm\
The $U_j$'s defined before the above proposition
in the case of the maximum are not independent under the
$P_n^{\mathcal{L}_\text{inv};\{\theta_j\}_{j=1}^n}$-type  Luce distributions.
Nor are they independent in the case of the  maximum or the minimum under the
$P_n^{\mathcal{L};\{\theta_j\}_{j=1}^n}$-type  Luce distributions.
One can check this with $n=3$, $\theta_i=\frac i6,\ i=1,2,3$.

\begin{proposition}\label{BrussresultMallows}
For each $n\in\mathbb{N}$, consider the secretary problem for the Mallows distribution $P_n^{\mathcal{M};q}$, $q\in(0,\infty)$, when either the largest number or the smallest number is of highest rank. For each of these problems, denote by $M_n^*$ an optimal strategy.
Then
\begin{equation}\label{BrussMallows}
\begin{aligned}
&\liminf_{n\to\infty}P_n^{\mathcal{M};q}(\mathcal{S}^\uparrow(n,M_n^*))\ge\frac1e;\\
&\liminf_{n\to\infty}P_n^{\mathcal{M};q}(\mathcal{S}^\downarrow(n,M_n^*))\ge\frac1e.
\end{aligned}
\end{equation}

\end{proposition}

We now consider the  precise  asymptotic behavior in detail.
In \cite{P22a}, we proved the following theorem  for the case that the largest number is  of highest rank, and the distribution is Mallows with $q\in(0,1)$.
\begin{thm*}{\bf (P-2)}
Consider the secretary problem
with the  Mallows distribution with $q=q_n\in(0,1)$ in the case that the largest number is  of highest rank..

\noindent i. Let $q_n=1-\frac cn$, where $c>0$. Then the asymptotically optimal strategy is $\mathcal{S}^\uparrow(n,M_n^*)$, where
\begin{equation}\label{optimali}
M_n^*\sim n\Big(\frac1c\log\big(1+\frac{e^c-1}e\big)\Big),
\end{equation}
 and the corresponding  limiting probability of success
is $\frac1e$:
$$
\lim_{n\to\infty}P_n^{\mathcal{M};q}(\mathcal{S}^\uparrow(n,M_n^*))=\frac1e.
$$
Also, $\lim_{c\to\infty} \frac1c\log\big(1+\frac{e^c-1}e\big)=1$ and $\lim_{c\to0} \frac1c\log\big(1+\frac{e^c-1}e\big)=\frac1e$.

\noindent ii.
Let $q_n=1-\frac c{n^\alpha}$, where $c>0$ and $\alpha\in(0,1)$. Then the asymptotically optimal strategy  is $\mathcal{S}^\uparrow(n,M_n^*)$, where
$$
n-M_n^*\sim\frac{n^\alpha}c,
$$
 and the corresponding limiting probability of success
is $\frac1e$:
$$
\lim_{n\to\infty}P_n^{\mathcal{M};q}(\mathcal{S}^\uparrow(n,M_n^*))=\frac1e.
$$
\noindent iii. Let $q\in(0,1)$.  Then the asymptotically optimal strategy  is $\mathcal{S}^\uparrow(n,M_n^*)$, where
$$
M_n^*=n-L, \ \text{where}\ L \ \text{satisfies}\ \frac q{1-q}\le L<\frac 1{1-q}.
$$
The corresponding limiting probability of success is given by
$$
\lim_{n\to\infty}P_n^{\mathcal{M};q}(\mathcal{S}^\uparrow(n,M_n^*))=(1-q)q^{L-1}L>\frac1e.
$$
%Also,
%$$
%\lim_{q\to1^-}\lim_{n\to\infty}P_n^{\mathcal{M};q}(\mathcal{S}^\uparrow(n,M_n^*))=\frac1e.
%$$
%and
%$$
%\lim_{n\to\infty}P_n^{\mathcal{M};q}(\mathcal{S}^\uparrow(n,M_n^*))=1-q, \ 0< q\le \frac12.
%$$
\end{thm*}
In this paper we will prove the following parallel result when $q>1$. The proof of the following theorem is along the same lines as the proof  of Theorem P-2, however the proofs of parts (ii) and (iii) are more involved than those  of parts (ii) and (iii) in Theorem P-2, and in part (iii) we don't have  a complete answer as we did in part (iii) of Theorem P-2.
\begin{theorem}\label{q>1}
Consider the secretary problem for a Mallows distribution with $q=q_n>1$ in the case that the largest number is  of highest rank.

\noindent i. Let $q_n=1+\frac cn$, where $c>0$. Then the asymptotically optimal strategy is $\mathcal{S}^\uparrow(n,M_n^*)$, where
\begin{equation}\label{optimaliq>1}
M_n^*\sim n\Big(\frac1c\log\big(1+\frac{1-e^{-c}}{e-1+e^{-c}}\big)\Big),
\end{equation}
 and the corresponding  limiting probability of success
is $\frac1e$:
$$
\lim_{n\to\infty}P_n^{\mathcal{M};q_n}(\mathcal{S}^\uparrow(n,M_n^*))=\frac1e.
$$
Also, $\lim_{c\to\infty} \frac1c\log\big(1+\frac{1-e^{-c}}{e-1+e^{-c}}\big)=0$ and $\lim_{c\to0} \frac1c\log\big(1+\frac{1-e^{-c}}{e-1+e^{-c}}\big)=\frac1e$.

\noindent ii.
Let $q_n=1+\frac c{n^\alpha}$, where $c>0$ and $\alpha\in(0,1)$. Then the asymptotically optimal strategy  is $\mathcal{S}^\uparrow(n,M_n^*)$, where
$$
M_n^*\sim\frac1c\left(1-\log(e-1)\right)n^\alpha\approx0.459\frac{n^\alpha}c,
$$
 and the corresponding limiting probability of success
is $\frac1e$:
$$
\lim_{n\to\infty}P_n^{\mathcal{M};q_n}(\mathcal{S}^\uparrow(n,M_n^*))=\frac1e.
$$
\noindent iii. Let $q>1$.  Then the asymptotically optimal strategy  is $\mathcal{S}^\uparrow(n,M^*)$, where
$$
M^*=\begin{cases}&\max\left\{M\ge1: \sum_{j=M}^\infty\frac1{q^j-1}\ge\frac q{q-1}\right\}, \ \text{if}\ \sum_{j=1}^\infty\frac1{q^j-1}\ge \frac q{q-1};\\
&0,\ \text{if}\ \sum_{j=1}^\infty\frac1{q^j-1}<\frac q{q-1}.\end{cases}
$$
(If  $\sum_{j=M^*}^\infty\frac1{q^j-1}=\frac q{q-1}$, then $M^*-1$ is also optimal.)

\noindent The corresponding limiting probability of success is given by
\begin{equation}\label{SuccessProbup}
\lim_{n\to\infty}P_n^{\mathcal{M};q}(\mathcal{S}^\downarrow(n,M^*))=\begin{cases}&\frac{q-1}q(1-q^{-M^*})\sum_{j=M^*}^\infty\frac1{q^j-1}\ge\\
&(1-q^{-M^*}),\ \text{if}\ M^*\ge1;\\
&\frac{q-1}q>,\ \text{if}\ M^*=0.\end{cases}
\end{equation}
The limiting probability in \eqref{SuccessProbup} is always greater or equal to $\frac1e$ and is strictly greater than $\frac1e$ for all but at most a countable number of $q$ with no accumulation points in $(0,\infty)$.
%In particular,
%$$
%\lim_{n\to\infty}P_n^{\mathcal{M};q}(\mathcal{S}^\uparrow(n,M_n^*))>\frac1e, \ q>1;
%$$
%$$
%\lim_{q\to1^+}\lim_{n\to\infty}P_n^{\mathcal{M};q}(\mathcal{S}^\uparrow(n,M_n^*))=\frac1e.
\end{theorem}
\bf \noindent Remark.\rm\ 
It is not hard to show that
the limiting optimal probability in part (iii) of Theorem P-2 is decreasing for $q\in(0,1)$. We expect that the limiting optimal probability in part (iii) of Theorem \ref{q>1}
is increasing in $q>1$, but have not been able to show it. In particular, if this is true, then the limiting probability is greater than $\frac1e$ for all $q>1$.

Theorem P-2 and Theorem  \ref{q>1} 
treat the asymptotic behavior in the case of a Mallows distribution when  the largest number is  of highest rank.
By Theorem \ref{secminmax} and Corollary \ref{LM}, the probability of success  in the case of the Mallows distribution with parameter $\frac1q\in(0,\infty)$ when the largest number is  of highest rank is equal to the probability of success in the case of the Mallows distributions
with parameter $q$ and in the case of  the $P_n^{\mathcal{L}_\text{inv};\{q_n^i\}_{i=1}^n}$-type Luce distribution
when  the smallest number is  of highest rank. From this along with the fact that
 if
$q_n=1\pm \frac c{n^\alpha}$, with $\alpha\in(0,1]$, then $\frac1{q_n}=1\mp\frac c{n^\alpha}+O(\frac1{n^{2\alpha}})$,
 the following theorem is essentially immediate.
\begin{theorem}
Consider the secretary problem
with either  the  Mallows distribution
or the $P_n^{\mathcal{L}_\text{inv};\{q_n^i\}_{i=1}^n}$-type Luce distribution with  $q=q_n\in(0,\infty)-\{1\}$
in the case that the smallest number is  of highest rank.

\noindent i. Let $q_n=1-\frac cn$, where $c>0$. Then the asymptotically optimal strategy is $\mathcal{S}^\downarrow(n,M_n^*)$, where
\begin{equation}\label{optimaliq>1}
M_n^*\sim n\Big(\frac1c\log\big(1+\frac{1-e^{-c}}{e-1+e^{-c}}\big)\Big),
\end{equation}
 and the corresponding  limiting probability of success
is $\frac1e$
%:$$
%\lim_{n\to\infty}P_n^{\mathcal{M};q_n}(\mathcal{S}^\downarrow(n,M_n^*))=\frac1e.
%$$

\noindent Also, $\lim_{c\to\infty} \frac1c\log\big(1+\frac{1-e^{-c}}{e-1+e^{-c}}\big)=0$ and $\lim_{c\to0} \frac1c\log\big(1+\frac{1-e^{-c}}{e-1+e^{-c}}\big)=\frac1e$.

\noindent ii.
Let $q_n=1-\frac c{n^\alpha}$, where $c>0$ and $\alpha\in(0,1)$. Then the asymptotically optimal strategy  is $\mathcal{S}^\downarrow(n,M_n^*)$, where
$$
M_n^*\sim\frac1c\left(1-\log(e-1)\right)n^\alpha\approx0.459\frac{n^\alpha}c,
$$
 and the corresponding limiting probability of success
is $\frac1e$.
%:$$
%\lim_{n\to\infty}P_n^{\mathcal{M};q_n}(\mathcal{S}^\downarrow(n,M_n^*))=\frac1e.
%$$

\noindent iii.
Let $q\in(0,1)$.  Then the asymptotically optimal strategy  is $\mathcal{S}^\downarrow(n,M^*)$, where
$$
M^*=\begin{cases}&\max\left\{M\ge1: \sum_{j=M}^\infty\frac{q^j}{1-q^j}\ge\frac 1{1-q}\right\}, \ \text{if}\ \sum_{j=1}^\infty\frac{q^j}{1-q^j}\ge\frac1{1-q};\\
&0,\ \text{if}\ \sum_{j=1}^\infty\frac{q^j}{1-q^j}<\frac1{1-q}.\end{cases}
$$
(If  $\sum_{j=M^*}^\infty\frac{q^j}{1-q^j}=\frac 1{1-q}$, then $M^*-1$ is also optimal.)

\noindent The corresponding limiting probability of success is given by
$$
\begin{cases}&(1-q)(1-q^{M^*})\sum_{j=M^*}^\infty\frac{q^j}{1-q^j}\ge(1-q^{-M^*}),\ \text{if}\ M^*\ge1;\\
&1-q>\frac1e,\ \text{if}\ M^*=0.\end{cases}
$$
This limiting probability is always greater or equal to $\frac1e$ and is
strictly greater than $\frac1e$ for all but at most a countable number of $q$ with no accumulation points in $(0,\infty)$.

\noindent iv. Let $q_n=1+\frac cn$, where $c>0$. Then the asymptotically optimal strategy is $\mathcal{S}^\downarrow(n,M_n^*)$, where
\begin{equation}\label{optimali}
M_n^*\sim n\Big(\frac1c\log\big(1+\frac{e^c-1}e\big)\Big),
\end{equation}
 and the corresponding  limiting probability of success
is $\frac1e$.
%:$$
%\lim_{n\to\infty}P_n^{\mathcal{M};q}(\mathcal{S}^\downarrow(n,M_n^*))=\frac1e.
%$$

\noindent Also, $\lim_{c\to\infty} \frac1c\log\big(1+\frac{e^c-1}e\big)=1$ and $\lim_{c\to0} \frac1c\log\big(1+\frac{e^c-1}e\big)=\frac1e$.

\noindent v.
Let $q_n=1+\frac c{n^\alpha}$, where $c>0$ and $\alpha\in(0,1)$. Then the asymptotically optimal strategy  is $\mathcal{S}^\downarrow(n,M_n^*)$, where
$$
n-M_n^*\sim\frac{n^\alpha}c,
$$
 and the corresponding limiting probability of success
is $\frac1e$.
%:$$
%\lim_{n\to\infty}P_n^{\mathcal{M};q}(\mathcal{S}^\downarrow(n,M_n^*))=\frac1e.
%$$

\noindent vi. Let $q>1$.  Then the asymptotically optimal strategy  is $\mathcal{S}^\downarrow(n,M_n^*)$, where
$$
M_n^*=n-L, \ \text{where}\ L \ \text{satisfies}\ \frac1{q-1}\le L<\frac q{q-1}.
$$
The corresponding limiting probability of success is given by
$$
\frac{q-1}{q^L}L>\frac1e.
$$
%Also,
%$$
%\lim_{q\to1^+}\lim_{n\to\infty}P_n^{\mathcal{M};q}(\mathcal{S}^\downarrow(n,M_n^*))=\frac1e.
%$$

\end{theorem}

We now consider the  $P_n^{\mathcal{L}_\text{inv};\{\theta_j\}_{j=1}^n}$-type Luce distributions with another natural set of weights. In order to describe the setup, we need to  state the following result, which as has already been mentioned,  will  play a crucial     role in the proof of Theorem \ref{lucesec}.
Recall that if $a_1\cdots a_n$ is a sequence of distinct real numbers, then its reduced permutation, denoted
by $\text{Red}(a_1\cdots a_n)$,  is the permutation
$\sigma\in S_n$ formed by replacing the smallest element with $\sigma_1$, the next smallest with $\sigma_2$, etc., and finally replacing the largest with $\sigma_n$.
\begin{thm*} {\bf(CDK)}
Let $\{V_j\}_{j=1}^n$ be independent exponential random variables with $V_j\sim\text{Exp}(\theta_j),\ j\in[n]$.
Then
\begin{equation}\label{Red}
P\left(\text{Red}(V_1V_2\cdots V_n)=\sigma\right)=P_n^{\mathcal{L}_\text{inv};\{\theta_j\}_{j=1}^n}(\sigma),
\end{equation}
where $\text{Red}(V_1V_2\cdots V_n)$ is the reduced permutation in $S_n$.
\end{thm*}
The above theorem is equivalent to Theorem 2.1 in \cite{CDK}.
.

Now let $\{V_j\}_{j=1}^n$ be IID standard exponential random variables: $V_j\sim\text{Exp}(1)$.
Denote their order statistics by $V_{(1)}\le V_{(2)}\le\cdots\le V_{(n)}$.
As noted in \cite{CDK}, the following result has an easy proof that can be found for example in \cite{Feller}:

\it\noindent The random variables $V_{(1)}, V_{(2)}-V_{(1)}, V_{(3)}-V_{(2)},\cdots V_{(n)}-V_{(n-1)}$
are independent random variables with distributions\rm\
\begin{equation}\label{orderstat}
V_{(1)}\sim\text{Exp}(n),\ V_{(2)}-V_{(1)}\sim\text{Exp}(n-1),\cdots, V_{(n)}-V_{(n-1)}\sim\text{Exp}(1).
\end{equation}
Consider now the permutation obtained by reducing the gaps between the order statistics:
$\text{Red}\left(V_{(1)}, V_{(2)}-V_{(1)}, V_{(3)}-V_{(2)},\cdots V_{(n)}-V_{(n-1)}\right)$.
From \eqref{Red} and \eqref{orderstat}, it follows that
\begin{equation}\label{gaps}
\begin{aligned}
& \text{Red}\left(V_{(1)}, V_{(2)}-V_{(1)}, V_{(3)}-V_{(2)},\cdots V_{(n)}-V_{(n-1)}\right)\in S_n
\ \text{has the }\\
& P_n^{\mathcal{L}_\text{inv};\{\theta^{(n)}_j\}_{j=1}^n}\ \text{-type Luce distribution where the weights are given by}\
\end{aligned}
\end{equation}
\begin{equation}\label{Sukh}
\theta^{(n)}_j=n+1-j,\ j\in[n].
\end{equation}
In \cite{CDK}, the weights in \eqref{Sukh} are called the \it Sukhatme weights\rm, named after the author who proved \eqref{gaps} \cite{S}.
We will call the weights $\theta_j=j,\ j\in[n]$, the \it reverse Sukhatme weights\rm.  From the above discussion, the reduced permutation
$\text{Red}\left( V_{(n)}-V_{(n-1)}, V_{(n-1)}-V_{(n-2)},\cdots,V_{(2)}-V_{(1)},V_{(1)}\right)$ has the
$P_n^{\mathcal{L}_\text{inv};\{\theta_j\}_{j=1}^n}$ Luce distribution with the reverse Sukhatme weights.

The following theorem analyzes  the secretary problem for the $P_n^{\mathcal{L}_\text{inv};\{\theta^{(n)}_j\}_{j=1}^n}$-type Luce distribution
with the Sukhatme weights  and with the reverse Sukhatme weights in the case that the smallest number is of highest rank.
\begin{theorem}\label{Suksec}
\noindent i. Consider the secretary problem for the
$P_n^{\mathcal{L}_\text{inv};\{\theta^{(n)}_j\}_{j=1}^n}$-type Luce distribution with the Sukhatme weights defined in \eqref{Sukh}
in the case that the smallest number is  of highest rank.
Then the asymptotically optimal strategy is $\mathcal{S}^\downarrow(n,M_n^*)$, where
\begin{equation}\label{Msukh}
M_n^*\sim \left(1-(1-\frac1e)^\frac12\right)n\approx0.2049\thinspace n,
\end{equation}
and the corresponding limiting probability of success is
$$
\lim_{n\to\infty}P_n^{\mathcal{L}_\text{inv};\{\theta^{(n)}_j\}_{j=1}^n}(\mathcal{S}^\downarrow(n,M_n^*))=\frac1e.
$$
\noindent ii. Consider the secretary problem for the
$P_n^{\mathcal{L}_\text{inv};\{\theta_j\}_{j=1}^n}$-type Luce distribution with the reverse Sukhatme weights $\theta_j=j,\ j\in[n]$
in the case that the smallest number is  of highest rank.
Then the asymptotically optimal strategy is $\mathcal{S}^\downarrow(n,M_n^*)$, where
\begin{equation}\label{Msukh}
M_n^*\sim e^{-\frac12}\thinspace n\approx0.6065\thinspace n,
\end{equation}
and the corresponding limiting probability of success is
$$
\lim_{n\to\infty}P_n^{\mathcal{L}_\text{inv};\{\theta_j\}_{j=1}^n}(\mathcal{S}^\downarrow(n,M_n^*))=\frac1e.
$$
\end{theorem}

In section \ref{prelim} we prove some  facts about sets of independent exponential random variables that  will be crucial  for the proofs of Theorem \ref{lucesec} and Proposition \ref{BrussresultLuce}.
We prove Theorem \ref{lucesec} in section \ref{pflucesec}, we prove Propositions \ref{BrussresultLuce} and \ref{BrussresultMallows} in section \ref{pfBrussresults}, we prove
Theorem \ref{q>1} in section \ref{pfq>1} and Theorem \ref{Suksec} in section \ref{pfSuksec}.

\section{Preliminary results concerning independent exponential random variables}\label{prelim}
In this, section the random variables  $\{V_i\}_{i=1}^\infty$ are independent with $V_i\sim\text{Exp}(\theta_i)$.
\begin{lemma}\label{1}
\begin{equation*}
P(V_1<V_2)=\frac{\theta_1}{\theta_1+\theta_2}.
\end{equation*}
\end{lemma}
\begin{proof}
$$
P(V_1<V_2)=\int_0^\infty dx\thinspace\theta_1e^{-\theta_1x}\int_{x_1}^\infty dy\thinspace\theta_2e^{-\theta_2y}=\frac{\theta_1}{\theta_1+\theta_2}.
$$
\end{proof}
\begin{lemma}\label{2}
\begin{equation*}\label{minexpdist}
\min(V_1,V_2,\cdots, V_r)\sim\text{\rm Exp}(\sum_{i=1}^r\theta_i),\ r\in\mathbb{N}.
\end{equation*}
\begin{proof}
$$
P(\min(V_1,V_2,\cdots, V_r)\ge x)=\prod_{i=1}^rP(V_i\ge x)=\prod_{i=1}^r e^{-\theta_ix}=e^{-(\sum_{i=1}^r\theta_i)x}.
$$
\end{proof}
\end{lemma}
\begin{lemma}\label{3}
\begin{equation*}\label{minexpconddist}
V_1|\{V_1=\min(V_1,\cdots, V_r)\}\sim\text{\rm Exp}(\sum_{i=1}^r\theta_i), \ r\in\mathbb{N}.
\end{equation*}
\end{lemma}
\begin{proof}
\begin{equation}\label{condform}
P(V_1\ge x|V_1=\min(V_1,\cdots, V_r))=\frac{P(V_1\ge x,V_1<\min(V_2,\cdots, V_r))}{P(V_1<\min(V_2,\cdots, V_r))}.
\end{equation}
From Lemmas \ref{1} and \ref{2},
\begin{equation}\label{denom}
P(V_1<\min(V_2,\cdots, V_r))=\frac{\theta_1}{\sum_{i=1}^r\theta_i}.
\end{equation}
From Lemma \ref{2},
\begin{equation}\label{numer}
\begin{aligned}
&P(V_1\ge x,V_1<\min(V_2,\cdots, V_r))=\int_x^\infty dx\thinspace\theta_1e^{-\theta_1x}\int_x^\infty dy\thinspace (\sum_{i=2}^r\theta_i)e^{-(\sum_{i=2}^r\theta_i)y}=\\
&\frac{\theta_1}{\sum_{i=1}^r\theta_i}e^{-(\sum_{i=1}^r\theta_i)x}.
\end{aligned}
\end{equation}
The lemma follows from \eqref{condform}-\eqref{numer}.
\end{proof}
\begin{lemma}\label{4}
\begin{equation}
V_1|\{V_1=\min(V_1,\cdots, V_r), V_2=\min(V_2,\cdots, V_s)\}  \sim\text{\rm Exp}(\sum_{i=1}^r\theta_i),\ 3\le s\le r.
\end{equation}
\end{lemma}
\begin{proof}
If $s<r$, we have
\begin{equation}\label{bigcondexp}
\begin{aligned}
&P(V_1\ge x|V_1=\min(V_1,\cdots, V_r), V_2=\min(V_2,\cdots, V_s) )=\\
&\frac{P(V_1\ge x,V_1<V_2, V_1<\min(V_{s+1},\cdots, V_r), V_2<\min(V_3,\cdots ,V_s))}{P(V_1<V_2, V_1<\min(V_{s+1},\cdots, V_r), V_2<\min(V_3,\cdots, V_s))}.
\end{aligned}
\end{equation}
If $s=r$, then the above equation holds but with the term $V_1<\min(V_{s+1},\cdots, V_r)$ deleted from the numerator and denominator.
We will assume from now on the $s<r$; the case $s=r$ follows similarly.
Let $Y=\min(V_3,\cdots V_s)$ and $Z=\min(V_{s+1},\cdots, V_r)$.
Then $V_1,V_2,Y,Z$ are independent and by Lemma \ref{2}, $Y\sim\text{Exp}(\sum_{i=3}^s\theta_i)$ and $Z\sim\text{Exp}(\sum_{i=s+1}^r\theta_i)$.
For notational convenience, denote the probability densities functions of these four random variables by $f_{V_1},f_{V_2},f_Y,f_Z$.

Then the denominator on the right hand side of \eqref{bigcondexp} is given by
\begin{equation}\label{keycalc}
P(V_1<V_2, V_1<\min(V_{s+1},\cdots, V_r), V_2<\min(V_3,\cdots V_s))=P(V_1<V_2,V_1<Z, V_2<Y),
\end{equation}
and
\begin{equation}\label{multiint}
\begin{aligned}
&P(V_1<V_2,V_1<Z, V_2<Y)=\\
&\int_0^\infty dv_2\int_{v_2}^\infty dy\left[\int_0^{v_2}dv_1\left(\int_{v_1}^\infty dz f_Z(z)\right)f_{V_1}(v_1)\right]f_Y(y)f_{V_2}(v_2).
\end{aligned}
\end{equation}
We have
$$
\int_{v_1}^\infty dz f_Z(z)=\exp(-(\sum_{i=s+1}^r\theta_i)v_1),
$$
and then
\begin{equation}\label{stage2}
\begin{aligned}
&\int_0^{v_2}dv_1\left(\int_{v_1}^\infty dz f_Z(z)\right)f_{V_1}(v_1)=\int_0^{v_2}\exp(-(\sum_{i=s+1}^r\theta_i)v_1)\thinspace \theta_1\exp(-\theta_1v_1)dv_1=\\
&\frac{\theta_1}{\theta_1+\sum_{i=s+1}^r\theta_i}\left(1-\exp\big(-(\theta_1+\sum_{i=s+1}^r\theta_i)v_2\big)\right).
\end{aligned}
\end{equation}
Therefore
\begin{equation}\label{stage3}
\begin{aligned}
&\int_{v_2}^\infty dy\left[\int_0^{v_2}dv_1\left(\int_{v_1}^\infty dz f_Z(z)\right)f_{V_1}(v_1)\right]f_Y(y)=\\
&\frac{\theta_1}{\theta_1+\sum_{i=s+1}^r\theta_i}\left(1-\exp(-(\theta_1+\sum_{i=s+1}^r\theta_i)v_2)\right)\int_{v_2}^\infty f_Y(y)dy=\\
&\frac{\theta_1}{\theta_1+\sum_{i=s+1}^r\theta_i}\left(1-\exp(-(\theta_1+\sum_{i=s+1}^r\theta_i)v_2)\right)\exp(-(\sum_{i=3}^s\theta_i)v_2),
\end{aligned}
\end{equation}
and
\begin{equation}\label{stage4}
\begin{aligned}
&\int_0^\infty dv_2\int_{v_2}^\infty dy\left[\int_0^{v_2}dv_1\left(\int_{v_1}^\infty dz f_Z(z)\right)f_{V_1}(v_1)\right]f_Y(y)f_{V_2}(v_2)=\\
&\frac{\theta_1\theta_2}{\theta_1+\sum_{i=s+1}^r\theta_i}\int_0^\infty\left(1-\exp(-(\theta_1+\sum_{i=s+1}^r\theta_i)v_2)\right)\exp\big(-(\sum_{i=3}^s\theta_i)v_2\big)\theta_2\exp(-\theta_2v_2)dv_2=\\
&\frac{\theta_1\theta_2}{\theta_1+\sum_{i=s+1}^r\theta_i}\left(\frac1{\sum_{i=2}^s\theta_i}-\frac1{\sum_{i=1}^r\theta_i}\right)=
\frac{\theta_1\theta_2}{(\sum_{i=2}^s\theta_i)(\sum_{i=1}^r\theta_i)}.
\end{aligned}
\end{equation}
From \eqref{multiint} and \eqref{stage4}, it follows that
\begin{equation}\label{uselatertoo}
P(V_1<V_2,V_1<Z, V_2<Y)=
\frac{\theta_1\theta_2}{(\sum_{i=2}^s\theta_i)(\sum_{i=1}^r\theta_i)}.
\end{equation}
We have displayed this formula because it will be  used in the proof of Theorem \ref{lucesec}.
By \eqref{keycalc}, the right hand side of \eqref{uselatertoo} is the denominator on the right hand side of \eqref{bigcondexp}.

The numerator on the right hand side of \eqref{bigcondexp} is calculated similarly.
We have
\begin{equation}\label{multiinttake2}
\begin{aligned}
&P(V_1\ge x,V_1<V_2, V_1<\min(V_{s+1},\cdots, V_r), V_2<\min(V_3,\cdots V_s))=\\
&P(V_1\ge x,V_1<V_2,V_1<Z, V_2<Y)=\\
&\int_x^\infty dv_2\int_{v_2}^\infty dy\left[\int_x^{v_2}dv_1\left(\int_{v_1}^\infty dz f_Z(z)\right)f_{V_1}(v_1)\right]f_Y(y)f_{V_2}(v_2).
\end{aligned}
\end{equation}
Instead of \eqref{stage2}, we have
\begin{equation}\label{stage2take2}
\begin{aligned}
&\int_x^{v_2}dv_1\left(\int_{v_1}^\infty dz f_Z(z)\right)f_{V_1}(v_1)=\int_x^{v_2}\exp(-(\sum_{i=s+1}^r\theta_i)v_1)\thinspace \theta_1\exp(-\theta_1v_1)dv_1=\\
&\frac{\theta_1}{\theta_1+\sum_{i=s+1}^r\theta_i}\left(\exp\big(-(\theta_1+\sum_{i=s+1}^r\theta_i)x\big)-\exp\big(-(\theta_1+\sum_{i=s+1}^r\theta_i)v_2\big)\right).
\end{aligned}
\end{equation}
And instead of \eqref{stage3} we have
\begin{equation}\label{stage3take2}
\begin{aligned}
&\int_{v_2}^\infty dy\left[\int_x^{v_2}dv_1\left(\int_{v_1}^\infty dz f_Z(z)\right)f_{V_1}(v_1)\right]f_Y(y)=\\
&\frac{\theta_1}{\theta_1+\sum_{i=s+1}^r\theta_i}\left(\exp(-(\theta_1+\sum_{i=s+1}^r\theta_i)x)-\exp(-(\theta_1+\sum_{i=s+1}^r\theta_i)v_2)\right)\exp(-(\sum_{i=3}^s\theta_i)v_2).
\end{aligned}
\end{equation}
And then instead of \eqref{stage4}, we have
\begin{equation}\label{stage4take2}
\begin{aligned}
&\int_x^\infty dv_2\int_{v_2}^\infty dy\left[\int_x^{v_2}dv_1\left(\int_{v_1}^\infty dz f_Z(z)\right)f_{V_1}(v_1)\right]f_Y(y)f_{V_2}(v_2)=\frac{\theta_1\theta_2}{\theta_1+\sum_{i=s+1}^r\theta_i}\times\\
&\int_x^\infty\left(\exp(-(\theta_1+\sum_{i=s+1}^r\theta_i)x)-\exp(-(\theta_1+\sum_{i=s+1}^r\theta_i)v_2)\right)\exp\big(-(\sum_{i=3}^s\theta_i)v_2\big)\theta_2\exp(-\theta_2)v_2)dv_2=\\
&\frac{\theta_1\theta_2}{\theta_1+\sum_{i=s+1}^r\theta_i}\left(\frac1{\sum_{i=2}^s\theta_i}-\frac1{\sum_{i=1}^r\theta_i}\right)\exp\big(-(\sum_{i=1}^r\theta_i)x\big)=\\
&\frac{\theta_1\theta_2}{(\sum_{i=2}^s\theta_i)(\sum_{i=1}^r\theta_i)}\exp\big(-(\sum_{i=1}^r\theta_i)x\big).
\end{aligned}
\end{equation}
The lemma now follows from \eqref{bigcondexp}, \eqref{keycalc}, \eqref{uselatertoo}   \eqref{multiinttake2}  and \eqref{stage4take2}.

\end{proof}
\section{Proof of Theorem \ref{lucesec}}\label{pflucesec}
Let  the smallest number 1 be of highest rank.
The event $\mathcal{S}^\downarrow(n,M))\subset S_n$ that the highest ranked item 1 is obtained via the
 strategy of
rejecting the first $M$ items and then  selecting the first
later-arriving item whose rank is higher than that of any of the first $M$ items (if such an item exists)
is given by
\begin{equation}\label{union}
\begin{aligned}
\mathcal{S}^\downarrow(n,M)=\begin{cases}\cup_{j=M+1}^n\left\{\sigma\in S_n:\sigma_j=1, \min(\sigma_1,\cdots\sigma_{j-1})=\min(\sigma_1\cdots,\sigma_M)\right\},\\
\text{if}\ M\in\{1,\cdots, n-1\};\\
 \{\sigma\in S_n:\sigma_1=1\},\ \text{if}\ M=0.\end{cases}
\end{aligned}
\end{equation}
From Remark 2 after Corollary \ref{LM}, it follows that  $P_n^{\mathcal{L}_\text{inv};\{\theta_j\}_{j=1}^n}(\sigma_1=1)=\frac{\theta_1}{\sum_{i=1}^n\theta_i}$.
(Remark 2 considers the case $\theta_i=q^i$, but the argument there works for any choice of $\theta_i$.)
This gives \eqref{lucesecmin} in the case $M=0$.

From now on, assume that $M\in\{1,\cdots, n-1\}$. From \eqref{union},
\begin{equation}\label{intermsofsigma}
P_n^{\mathcal{L}_\text{inv};\{\theta_j\}_{j=1}^n}(\mathcal{S}^\downarrow(n,M))=\sum_{j=M+1}^n
P_n^{\mathcal{L}_\text{inv};\{\theta_j\}_{j=1}^n}(\sigma_j=1, \min(\sigma_1,\cdots\sigma_{j-1})=\min(\sigma_1\cdots,\sigma_M)).
\end{equation}
 Let $\{V_j\}_{j=1}^n$ be independent exponential random variables with $V_j\sim\text{Exp}(\theta_j),\ j\in[n]$.
Then it follows from \eqref{Red} in Theorem CDK that
\begin{equation}\label{mins}
\begin{aligned}
&P_n^{\mathcal{L}_\text{inv};\{\theta_j\}_{j=1}^n}(\sigma_j=1, \min(\sigma_1,\cdots\sigma_{j-1})=\min(\sigma_1\cdots,\sigma_M))=\\
&\begin{cases}P(V_j=\min_{i\in[n]}V_i,\min_{i\in[j-1]}V_i=\min_{i\in[M]}V_i),\ j\in\{M+2,\cdots n\};\\P(V_{M+1}=\min_{i\in[n]}V_i),\ j=M+1.\end{cases}
\end{aligned}
\end{equation}
We now calculate the probability on the right hand side of \eqref{mins}. We will assume that  $j\in\{M+2,\cdots n\}$; the case $j=M+1$ follows similarly.
Let
$$
X=\min(V_1,\cdots, V_M);\ \ Y=\min(\sigma_{M+1},\cdots\sigma_{j-1});\ \ Z=\min(V_{j+1},\cdots, V_n);\ \  W=V_j.
$$
Note that $X,Y,Z$ and $W$ are independent. Also, by Lemma \ref{2},
$$
X\sim\text{Exp}(\sum_{i=1}^M\theta_i),\ \ Y\sim\text{Exp}(\sum_{i=M+1}^{j-1}\theta_i),\ \ Z\sim\text{Exp}(\sum_{i=j+1}^n\theta_i).
$$
Then
\begin{equation}\label{XYZW}
\begin{aligned}
&P(V_j=\min_{i\in[n]}V_i,\min_{i\in[j-1]}V_i=\min_{i\in[M]}V_i)=P(Y>X, W<\min(X,Y,Z))=\\
&P(W<X,W<Z,X<Y).
\end{aligned}
\end{equation}
The above probability was calculated in
\eqref{uselatertoo}.
In that formula, $V_1$ plays the role of $W$ here and $V_2$ plays the role of $X$ here.
In \eqref{uselatertoo}  $V_i\sim\text{Exp}(\theta_i),\ i=1,2,$ whereas here $W\sim\text{Exp}(\theta_j)$ and $X\sim\text{Exp}(\sum_{i=1}^M\theta_i)$. Also,
in \eqref{uselatertoo} $Y$ is distributed exponentially with parameter $\sum_{i=3}^s\theta_i$ whereas here $Y$ is  distributed exponentially with parameter $\sum_{i=M+1}^{j-1}\theta_i$,
and in \eqref{uselatertoo} $Z$ is distributed exponentially with parameter $\sum_{i=s+1}^r\theta_i$ whereas here $Z$ is distributed exponentially with parameter $\sum_{i=j+1}^n\theta_i$.
Therefore, making the necessary changes, we obtain from \eqref{uselatertoo} that
\begin{equation}\label{laternow}
P(W<X,W<Z,X<Y)=
\frac{\theta_j\sum_{i=1}^M\theta_i}{(\sum_{i=1}^{j-1}\theta_i)(\sum_{i=1}^n\theta_i)}.
\end{equation}
Now \eqref{lucesecmin} when $M\in\{1,\cdots, n-1\}$ follows from \eqref{intermsofsigma}-\eqref{laternow}.
\hfill $\square$

\section{Proofs of Propositions \ref{BrussresultLuce} and \ref{BrussresultMallows}}\label{pfBrussresults}
\noindent \it Proof of Proposition \ref{BrussresultLuce}.\rm\
By the paragraph before the statement of the proposition, it suffices to show that
the random variables $\{U_j\}_{j=1}^n$ are independent, where
 $U_j=1_{\{\sigma_j=\min(\sigma_1\cdots \sigma_j)\}}$.
Since $U_1=1$ deterministically, it suffices to show that
\begin{equation}\label{fullindep}
\begin{aligned}
&P_n^{\mathcal{L}_\text{inv};\{\theta_j\}_{j=1}^n}(U_{l_1}=U_{l_2}=\cdots=U_{l_r}=1)=\prod_{k=1}^rP_n^{\mathcal{L}_\text{inv};\{\theta_j\}_{j=1}^n}(U_{l_k}=1),\\
&2\le l_1<\cdots<l_r\le n.
\end{aligned}
\end{equation}

 Let $\{V_j\}_{j=1}^n$ be independent exponential random variables with $V_j\sim\text{Exp}(\theta_j),\ j\in[n]$.
By \eqref{Red} and Lemmas \ref{1} and  \ref{2},
\begin{equation}\label{Ukalone}
P_n^{\mathcal{L}_\text{inv};\{\theta_j\}_{j=1}^n}(U_{l_k}=1)=P(V_{l_k}=\min_{i\in[l_k]}V_i)=P(V_{l_k}<\min_{i\in[l_k-1]}V_i)=\frac{\theta_{l_k}}{\sum_{i=1}^{l_k}\theta_i}.
\end{equation}
By \eqref{Red},
\begin{equation}\label{Uktogether}
\begin{aligned}
&P_n^{\mathcal{L}_\text{inv};\{\theta_j\}_{j=1}^n}(U_{l_1}=U_{l_2}=\cdots=U_{l_r}=1)=\\
&P(V_{l_1}=\min_{1\le i\le l_1}V_i,V_{l_2}=\min_{1\le i\le l_2}V_i,\cdots, V_{l_r}=\min_{1\le i\le l_r}V_i).
\end{aligned}
\end{equation}

Let
$$
\begin{aligned}
&Y_1=\min(V_1,\cdots, V_{l_1-1}), \ Y_2=\min(V_{l_1+1},\cdots, V_{l_2-1}),\cdots,\\
& Y_r=\min(V_{l_{r-1}+1},\cdots, V_{l_r-1}).
\end{aligned}
$$
We will assume that $l_k-l_{k-1}\ge2$, for $k=2,\cdots, r$. Otherwise not all of the above random variables are well defined.
When this condition is not met, the proof proceeds similarly, removing the random variables that aren't defined.
Note that the random variables $\{Y_i\}_{i=1}^r$ and $\{V_{l_i}\}_{i=1}^r$
are all independent, and by Lemma \ref{2},
$$
Y_k\sim\text{Exp}(\sum_{i=l_{k-1}+1}^{l_k-1}\theta_i),\ k=1,\cdots,r,
$$
where we define $l_0=0$.
We can rewrite
\eqref{Uktogether} as
\begin{equation}\label{Uktogetheragain}
\begin{aligned}
&P_n^{\mathcal{L}_\text{inv};\{\theta_j\}_{j=1}^n}(U_{l_1}=U_{l_2}=\cdots=U_{l_r}=1)=\\
&P(V_{l_1}<Y_1, V_{l_2}<V_{l_1}, V_{l_2}<Y_2,\cdots, V_{l_{r-1}}<V_{l_{r-2}}, V_{l_{r-1}}<Y_{r-1},V_{l_r}<V_{l_{r-1}}, V_{l_r}<Y_r).
\end{aligned}
\end{equation}

%&P(V_{l_1}<Y_1, V_{l_2}<V_{l_1}, V_{l_2}<Y_2,\cdots, V_{l_{r-1}}<V_{l_{r-2}}, V_{l_{r-1}}<Y_{r-1})\times\\
%&P(V_{l_r}<V_{l_{r-1}}, V_{l_r}<Y_r|V_{l_1}<Y_1, V_{l_2}<V_{l_1}, V_{l_2}<Y_2,\cdots, V_{l_{r-1}}<V_{l_{r-2}}, V_{l_{r-1}}<Y_{r-1}).

 Let
 $$
 A_k=\{V_{l_1}<Y_1, V_{l_2}<V_{l_1}, V_{l_2}<Y_2,\cdots, V_{l_k}<V_{l_{k-1}}, V_{l_k}<Y_k\},\ k=1\cdots,r.
 $$
 Note that $A_k\subset\{V_{l_k}=\min(V_1,\cdots, V_{l_k})\}$.
We rewrite  the right hand side  of \eqref{Uktogetheragain} as
\begin{equation}\label{condwithA}
\begin{aligned}
&P(V_{l_1}<Y_1,\cdots, V_{l_{r-1}}<V_{l_{r-2}}, V_{l_{r-1}}<Y_{r-1},V_{l_r}<V_{l_{r-1}}, V_{l_r}<Y_r)=\\
&P(A_r)=P(A_{r-1})P(V_{l_r}<V_{l_{r-1}}, V_{l_r}<Y_r|A_{r-1}).
\end{aligned}
\end{equation}
We have $A_1=\{V_{l_1}=\min(V_1,\cdots, V_{l_1})\}$. Thus, by Lemma \ref{3}, the conditional distribution of $V_{l_1}$ given $A_1$ is
$V_{l_1}|A_1\sim \text{Exp}(\sum_{i=1}^{l_1}\theta_i)$.
We have $A_2=\{V_{l_2}=\min(V_1,\cdots, V_{l_2}), V_{l_1}=\min(V_1,\cdots, V_{l_1})\}$.
Thus, by Lemma \ref{4}, the conditional distribution of $V_{l_2}$ given $A_2$ is
$V_{l_2}|A_2\sim \text{Exp}(\sum_{i=1}^{l_2}\theta_i)$.
Lemma \ref{4} can be extended in the obvious way (with a more and more tedious proof) to give
$V_{l_k}|A_k\sim\text{Exp}(\sum_{i=1}^{l_k}\theta_i)$, for $k=1,\cdots, r$.
Consequently, the second probability on the right hand side of \eqref{condwithA} satisfies
\begin{equation}\label{simplecond}
\begin{aligned}
&P(V_{l_r}<V_{l_{r-1}}, V_{l_r}<Y_r|A_{r-1})=P(V_{l_r}<W, V_{l_r}<Y_r)=\\
& P(V_{l_r}<\min(W,Y_r)),
\text{where}\ W\sim\text{Exp}(\sum_{i=1}^{l_{r-1}}\theta_i)\    \text{is independent of}\ V_{l_r}\ \text{and}\ Y_r.
\end{aligned}
\end{equation}
Thus, by Lemmas \ref{1} and \ref{2},
it follows from \eqref{simplecond} that
\begin{equation}\label{finalcondform}
P(V_{l_r}<V_{l_{r-1}}, V_{l_r}<Y_r|A_{r-1})=\frac{\theta_{l_r}}{\sum_{i=1}^{l_r}\theta_i}.
\end{equation}
From \eqref{Uktogetheragain}, \eqref{condwithA} and \eqref{finalcondform}, we conclude that
\begin{equation}\label{jointprob}
P_n^{\mathcal{L}_\text{inv};\{\theta_j\}_{j=1}^n}(U_{l_1}=U_{l_2}=\cdots=U_{l_r}=1)=\prod_{k=1}^r\frac{\theta_{l_k}}{\sum_{i=1}^{l_k}\theta_i}.
\end{equation}
Now \eqref{fullindep} follows from   \eqref{Ukalone} and \eqref{jointprob}.
\hfill $\square$

\

\noindent \it Proof of Proposition \ref{BrussresultMallows}.\rm\
Consider first the case that the smallest number is of highest rank.
By the paragraph before the statement of Proposition \ref{BrussresultLuce}, it suffices to show that
the random variables $\{U_j\}_{j=1}^n$ are independent, where
 $U_j=1_{\{\sigma_j=\min(\sigma_1\cdots \sigma_j)\}}$.
Let $\tilde U_j=1_{\{\sigma^{-1}_j=\min(\sigma^{-1}_1\cdots \sigma^{-1}_j)\}}$.
Since the Mallows distributions are invariant under the map $\sigma\to\sigma^{-1}$, it suffices to show that
$\{\tilde U_j\}_{j=1}^n$ are independent. This can be shown easily using the second of the online constructions that we presented for the Mallows distributions.
The random variable $\tilde U_1$ is deterministically equal to 1. For $j\in\{2,\cdots, n\}$,
the event $\{\tilde U_j=1\}$ is the event that the number $j$ appears to the left of all the numbers $\{1,\cdots, j-1\}$.
Thus, by the second online construction, the event  $\{\tilde U_j=1\}$  occurs if and only if $X_j=j-1$, where
$X_j$ is given in \eqref{X-inv}.
Since the $\{X_j\}_{j=2}^n$ are independent, it follows that $\{\tilde U_j\}_{j=1}^n$ are independent.

The proof in the case that the largest number is of highest rank is carried out similarly.
Now $U_j=1_{\{\sigma_j=\max(\sigma_1\cdots \sigma_j)\}}$ and
 $\tilde U_j=1_{\{\sigma^{-1}_j=\max(\sigma^{-1}_1\cdots \sigma^{-1}_j)\}}$.
The event $\{\tilde U_j=1\}$ is the event that the number $j$ appears to the right of all the numbers $\{1,\cdots, j-1\}$,
which occurs if and only if $X_j=0$.
\hfill $\square$

\section{Proof of Theorem \ref{q>1}}\label{pfq>1}
The proof of the theorem follows along similar lines  to the proof of Theorem P-2 in \cite{P22a}, however parts (ii) and (iii) here are considerably more delicate than  parts (ii) and (iii)  of Theorem P-2.

\noindent\it Proof of part (i).\rm\ We substitute $q=1+\frac cn$ and $M=M_n\sim bn$, where $b\in(0,1)$, in \eqref{exactform}. This gives
\begin{equation}\label{cnbn}
\begin{aligned}
&P_n^{\mathcal{M};q}(\mathcal{S}^\downarrow(n,M_n))\sim\\
&\frac cn\thinspace\frac1{(1+\frac cn)^n-1}(1+\frac cn)^{(1-b)n}\left((1+\frac cn)^{bn}-1\right)\sum_{j=[bn]+1}^n\frac1{(1+\frac cn)^{j-1}-1}\sim\\
&\frac c{e^c-1}e^{c(1-b)}(e^{cb}-1)\frac1n\sum_{j=[bn]+1}^n\frac1{e^{\frac{c(j-1)}n}-1}.
\end{aligned}
\end{equation}
We have
\begin{equation}\label{integrallog}
\begin{aligned}
&\lim_{n\to\infty}\frac1n\sum_{j=[bn]+1}^n\frac1{e^{\frac{c(j-1)}n}-1}=\int_b^1\frac1{e^{cx}-1}dx=\int_b^1\frac{e^{-cx}}{1-e^{-cx}}dx=\\
&\frac1c\log(1-e^{-cx})|_b^1=\frac1c\log\frac{1-e^{-c}}{1-e^{-bc}}.
\end{aligned}
\end{equation}
From \eqref{cnbn} and \eqref{integrallog} we conclude that
\begin{equation}\label{cncase}
\begin{aligned}
&\lim_{n\to\infty}P_n^{\mathcal{M};q}(\mathcal{S}^\downarrow(n,M_n))=\frac c{e^c-1}e^{c(1-b)}(e^{cb}-1)\frac1c\log\frac{1-e^{-c}}{1-e^{-bc}}=\\
&\frac{1-e^{-bc}}{1-e^{-c}}\log\frac{1-e^{-c}}{1-e^{-cb}},\ \text{if}\ M_n\sim bn,\ b\in(0,1).
\end{aligned}
\end{equation}
Let $H(b)=\frac{1-e^{-bc}}{1-e^{-c}}\log\frac{1-e^{-c}}{1-e^{-cb}}$. Note that $H(0^+)=H(1)=0$ and $H$ is nonnegative.
Differentiating, one readily finds that $H'(b)$ vanishes if and only if $\frac{1-e^{-c}}{1-e^{-bc}}=e$.
The unique $b$ solving this equation, which we denote by $b^*$, is given by
\begin{equation}\label{b*}
b^*=-\frac1c\log(1-e^{-1}+e^{-c-1})=\frac1c\log(1+\frac{1-e^{-c}}{e-1+e^{-c}}).
\end{equation}
Thus, $H$ attains its maximum at $b^*$. Since $\frac{1-e^{-c}}{1-e^{-b^*c}}=e$, we have $H(b^*)=\frac1e$. Using this with \eqref{cncase},
we  conclude that the optimal strategy is $M=M_n^*\sim n\left(\frac1c\log(1+\frac{1-e^{-c}}{e-1+e^{-c}})\right)$, and
$\lim_{n\to\infty}P_n^{\mathcal{M};q}(\mathcal{S}^\downarrow(n,M^*_n))=\frac1e$.
Finally, it is easy to show that
 $\lim_{c\to\infty} \frac1c\log\big(1+\frac{1-e^{-c}}{e-1+e^{-c}}\big)=0$ and $\lim_{c\to0} \frac1c\log\big(1+\frac{1-e^{-c}}{e-1+e^{-c}}\big)=\frac1e$.
\hfill $\square$

\

\noindent \it Proof of part (ii).\rm\
Substituting $q=1+\frac c{n^\alpha}$ and $M=M_n$ in \eqref{exactform}, we obtain
\begin{equation}\label{cnalpha}
\begin{aligned}
&P_n^{\mathcal{M};q}(\mathcal{S}^\downarrow(n,M_n))=\\
&\frac c{n^\alpha}\frac1{(1+\frac c{n^\alpha})^n-1}(1+\frac c{n^\alpha})^{n-M_n-1}\left((1+\frac c{n^\alpha})^{M_n}-1\right)\sum_{j=M_n+1}^n\frac1{(1+\frac c{n^\alpha})^{j-1}-1}.
\end{aligned}
\end{equation}
First assume that $M_n$ satisfies $n^\alpha=o(M_n)$. Since $\lim_{n\to\infty}(1+\frac c{n^\alpha})^n=\infty$, we have
\begin{equation}\label{easyupper}
\frac1{(1+\frac c{n^\alpha})^n-1}(1+\frac c{n^\alpha})^{n-M_n-1}\left((1+\frac c{n^\alpha})^{M_n}-1\right)\le 2,\ \text{for sufficiently large}\ n.
\end{equation}
Since $(1+\frac1x)^x$ is increasing in $x>0$ and is equal to $\frac94$   when $x=2$, we certainly have
$(1+\frac c{n^\alpha})^{j-1}-1\ge 2^\frac{c(j-1)}{n^\alpha}$ for sufficiently large $n$.
Thus,
\begin{equation}\label{est}
\sum_{j=M_n+1}^n\frac1{(1+\frac c{n^\alpha})^{j-1}-1}\le\sum_{j=M_n+1}^\infty(\frac12)^\frac{c(j-1)}{n^\alpha}=\frac{(\frac12)^{\frac{cM_n}{n^\alpha}}}{1-(\frac12)^\frac c{n^\alpha}}.
\end{equation}
Since $1-(\frac12)^\frac c{n^\alpha}=1-e^{-\frac c{n^\alpha}\log 2}=O(n^{-\alpha})$ and since $n^\alpha=o(M_n)$, it follows from \eqref{cnalpha}-\eqref{est} that
$\lim_{n\to\infty}P_n^{\mathcal{M};q}(\mathcal{S}^\downarrow(n,M_n))=0$.
Thus, the optimal $M_n$ satisfies $M_n=O(n^\alpha)$.

Now assume $M_n\sim bn^\alpha$, where $b\in(0,\infty)$.
Then we have
\begin{equation}\label{Mnbn}
\begin{aligned}
&P_n^{\mathcal{M};q}(\mathcal{S}^\downarrow(n,M_n))=\\
&\frac c{n^\alpha}\frac1{(1+\frac c{n^\alpha})^n-1}(1+\frac c{n^\alpha})^{n-M_n-1}\left((1+\frac c{n^\alpha})^{M_n}-1\right)\sum_{j=M_n+1}^n\frac1{(1+\frac c{n^\alpha})^{j-1}-1}=\\
&\frac c{n^\alpha}\frac{(1+\frac c{n^\alpha})^{n-1}}{(1+\frac c{n^\alpha})^n-1}\frac{(1+\frac c{n^\alpha})^{M_n}-1}{(1+\frac c{n^\alpha})^{M_n}}
\sum_{j=M_n+1}^n\frac1{(1+\frac c{n^\alpha})^{j-1}-1}\sim\\
&\frac c{n^\alpha}\frac{e^{cb}-1}{e^{cb}}\sum_{j=M_n+1}^n\frac1{(1+\frac c{n^\alpha})^{j-1}-1}\sim\frac c{n^\alpha}\frac{e^{cb}-1}{e^{cb}}\sum_{j=[bn^\alpha]+1}^n\frac1{(1+\frac c{n^\alpha})^{j-1}-1}.
\end{aligned}
\end{equation}
The transition on the final line above follows because  $\lim_{n\to\infty}\sum_{j=M_n+1}^n\frac1{(1+\frac c{n^\alpha})^{j-1}-1}=\infty$, which is rather easy to see and which in any case will come out from what follows below.
We also have
\begin{equation}\label{sumsim}
\sum_{j=[bn]+1}^n\frac1{(1+\frac c{n^\alpha})^{j-1}-1}\sim\sum_{j=[bn^\alpha]+1}^n\frac1{e^{\frac{c(j-1)}{n^\alpha}}-1}.
\end{equation}
In order not to interrupt the main track of the proof, we postpone the proof of \eqref{sumsim} until the end.

Let $\epsilon>0$. Writing
\begin{equation}\label{sumanalysis1}
\begin{aligned}
&\sum_{j=[bn^\alpha]+1}^n\frac1{e^{\frac{c(j-1)}{n^\alpha}}-1}=\left(\frac1{e^{\frac{c[bn^\alpha]}{n^\alpha}}-1}+\frac1{e^{\frac{c([bn^\alpha]+1)}{n^\alpha}}-1}+\cdots
\frac1{e^{\frac{c([bn^\alpha]+[\epsilon n^\alpha]-1)}{n^\alpha}}-1}\right)+\\
&\left(\frac1{e^{\frac{c([bn^\alpha]+[\epsilon n^\alpha])}{n^\alpha}}-1}+\cdots
\frac1{e^{\frac{c([bn^\alpha]+2[\epsilon n^\alpha]-1)}{n^\alpha}}-1}\right)+\cdots,
\end{aligned}
\end{equation}
and noting that
\begin{equation}\label{sumanalysis2}
\begin{aligned}
&[\epsilon n^\alpha]\frac1{e^{\frac{c([bn^\alpha]+(l+1)[\epsilon n^\alpha]-1)}{n^\alpha}}-1}\le
\left(\frac1{e^{\frac{c([bn^\alpha]+l[\epsilon n^\alpha])}{n^\alpha}}-1}+\cdots
\frac1{e^{\frac{c([bn^\alpha]+(l+1)[\epsilon n^\alpha]-1)}{n^\alpha}}-1}\right)\le\\
&[\epsilon n^\alpha]\frac1{e^{\frac{c([bn^\alpha]+l[\epsilon n^\alpha])}{n^\alpha}}-1},\ l=0,1,\cdots, \\
\end{aligned}
\end{equation}
it follows upon letting $n\to\infty$ and then letting $\epsilon\to0$ that
\begin{equation}\label{sumanalysisanswer}
\begin{aligned}
&\lim_{n\to\infty}\frac1{n^\alpha}\sum_{j=[bn^\alpha]+1}^n\frac1{e^{\frac{c(j-1)}{n^\alpha}}-1}=\int_b^\infty\frac1{e^{cx}-1}dx=\int_b^\infty\frac{e^{-cx}}{1-e^{-cx}}dx=
-\frac1c\log(1-e^{-cb}).
\end{aligned}
\end{equation}
From \eqref{Mnbn} and  \eqref{sumanalysisanswer}, we conclude that
\begin{equation}\label{formulaforb}
\lim_{n\to\infty}P_n^{\mathcal{M};q}(\mathcal{S}^\downarrow(n,M_n))=-(1-e^{-cb})\log(1-e^{-cb}).
\end{equation}
The function $-x\log x$, for $x\in(0,1)$, attains its maximum value of $\frac1e$  at $x=\frac1e$.
Thus, the righthand side of \eqref{formulaforb} as a function of $b\in(0,\infty)$ attains its maximum value $\frac1e$  when $1-e^{-cb}=\frac1e$, or equivalently,
$b=\frac1c\left(1-\log(e-1)\right)$.
This proves part (ii).

We now go back to prove \eqref{sumsim}.
It is clear that for any $B>b$,
\begin{equation}\label{simbB}
\sum_{j=M_n+1}^{[Bn^\alpha]}\frac1{(1+\frac c{n^\alpha})^{j-1}-1}\sim\sum_{j=[bn^\alpha]+1}^{[Bn^\alpha]}\frac1{e^{\frac{c(j-1)}{n^\alpha}}-1}.
\end{equation}
Also,
\begin{equation}\label{lower}
\begin{aligned}
&\sum_{j=[bn^\alpha]+1}^{[Bn^\alpha]}\frac1{(1+\frac c{n^\alpha})^{j-1}-1}\ge([(b+1)n^\alpha]-[bn^\alpha])e^{-(b+1)c}\ge\\
& (n^\alpha-2)e^{-(b+1)c},\  B\ge b+1.
\end{aligned}
\end{equation}
And for an appropriate constant $K$, independent of $B\ge1$,
\begin{equation}\label{upper}
\begin{aligned}
&\sum_{j=[Bn^\alpha]+1}^n\frac1{(1+\frac c{n^\alpha})^{j-1}-1}\le K\sum_{j=[Bn^\alpha]+1}^\infty\left(1+\frac c{n^\alpha}\right)^{1-j}=\\
&K\left(1+\frac c{n^\alpha}\right)^{-[Bn^\alpha]}\frac{n^\alpha(1+\frac c{n^\alpha})}c\sim\frac Kce^{-Bc}n^\alpha.
\end{aligned}
\end{equation}
Now \eqref{sumsim} follows from \eqref{simbB}-\eqref{upper} and the fact that $B$ can be chosen arbitrarily large.
\hfill $\square$
\

\noindent \it Proof of part (iii).\rm\ We fix $q>1$. First assume that $M=M_n\to\infty$.
From \eqref{exactform}
we have for sufficiently large $n$,
\begin{equation*}
\begin{aligned}
&P_n^{\mathcal{M};q}(\mathcal{S}^\uparrow(n,M_n))=
\frac{q-1}{q^n-1}q^{n-M_n-1}(q^{M_n}-1)\sum_{j=M_n+1}^n\frac1{q^{j-1}-1}\le\\
&\frac{q-1}q\sum_{j=M_n+1}^n\frac1{q^{j-1}-1}\stackrel{n\to\infty}{\rightarrow0}.
\end{aligned}
\end{equation*}
Thus, the optimal strategy must be some fixed $M$.  From \eqref{exactform}, we have
\begin{equation}\label{fixedqprob}
\begin{aligned}
&\lim_{n\to\infty}P_n^{\mathcal{M};q}(\mathcal{S}^\uparrow(n,M))=
\lim_{n\to\infty}\frac{q-1}{q^n-1}q^{n-M-1}(q^M-1)\sum_{j=M+1}^n\frac1{q^{j-1}-1}=\\
&\frac{q-1}q(1-q^{-M})\sum_{j=M}^\infty\frac1{q^j-1},\ \text{if}\ M\ge1,\\
&\lim_{n\to\infty}P_n^{\mathcal{M};q}(\mathcal{S}^\uparrow(n,M))=\frac{q-1}q,\ \text{if}\ M=0.
\end{aligned}
\end{equation}

Let
\begin{equation}\label{G}
G(M)=\begin{cases}(1-q^{-M})\sum_{j=M}^\infty\frac1{q^j-1},\ \text{if}\ M\ge1;\\ 1,\ \text{if}\  M=0.\end{cases}
\end{equation}
We have
\begin{equation*}\label{Gdifference}
\begin{aligned}
&G(M)-G(M-1)=(1-q^{-M})\sum_{j=M}^\infty\frac1{q^j-1}-(1-q^{-M+1})\sum_{j=M-1}^\infty\frac1{q^j-1}=\\
&(q^{-M+1}-q^{-M})\sum_{j=M}^\infty\frac1{q^j-1}-(1-q^{-M+1})\frac1{q^{M-1}-1}=\\
&q^{-M+1}\left((1-\frac1q)\sum_{j=M}^\infty\frac1{q^j-1}-1\right), \ M\ge2.
\end{aligned}
\end{equation*}
Also,
$$
G(1)-G(0)=(1-\frac1q)\sum_{j=1}^\infty\frac1{q^j-1}-1.
$$
Consequently,
\begin{equation}\label{Gdiff}
G(M)\ge G(M-1)\ \ \text{if and only if}\ \sum_{j=M}^\infty\frac1{q^j-1}\ge\frac q{q-1},\ \text{for}\  M\ge1.
\end{equation}
Thus,
$G$ is unimodal and $G$ attains its maximum at
\begin{equation}\label{M*}
M^*:=\begin{cases}\max\{M\ge1: \sum_{j=M}^\infty\frac1{q^j-1}\ge\frac q{q-1}\},\ \text{if}\ \sum_{j=1}^\infty\frac1{q^j-1}\ge\frac q{q-1};\\ 0,\ \text{if}\ \sum_{j=1}^\infty\frac1{q^j-1}<\frac q{q-1}.\end{cases}
\end{equation}
(If $\sum_{M^*}^\infty\frac1{q^j-1}=\frac q{q-1}$, then the maximum is also attained at $M^*-1$.)
From \eqref{G}, \eqref{M*} and  \eqref{fixedqprob}
we have
\begin{equation}\label{finalfixedq}
\lim_{n\to\infty}P_n^{\mathcal{M};q}(\mathcal{S}^\uparrow(n,M^*))=\begin{cases}\frac{q-1}q(1-q^{-M^*})\sum_{j=M^*}^\infty\frac1{q^j-1}\ge1-q^{-M^*},\ \text{if}\ M^*\ge1;\\
\frac{q-1}q,\ \text{if}\ M^*=0.\end{cases}
\end{equation}

 By Proposition \ref{BrussresultLuce},
the right hand side of \eqref{finalfixedq} is greater or equal to $\frac1e$.
It remains to show that it is strictly greater than $\frac1e$, except at no more than a countable number of $q\in(1,\infty)$ with no accumulations points in $(1,\infty)$.
For any $M\in\mathbb{N}$, the function $\frac{q-1}q\sum_{j=M}^\infty\frac1{q^j-1}$ is decreasing for $q>1$ and converges to $\infty$ when $q\to1^+$.
This can be seen be writing the function in the form $\frac1q\sum_{j=M}^\infty\frac1{1+q+\cdots+q^{j-1}}$.
Of course, $\lim_{M\to\infty}\frac{q-1}q\sum_{j=M}^\infty\frac1{q^j-1}=0$, for each $q>1$.
%One can check that $\frac{q-1}q\sum_{j=1}^\infty\frac1{q^j-1}=1$ when $q\approx1.734$.
From these facts and \eqref{M*} it follows that there exists a sequence $\{q_M\}_{M=1}^\infty$ that is decreasing and converges to 1,
%with $q_1\approx1.74$,
such that
$$
M^*=M^*(q)=\begin{cases} M,\ \text{if}\ q_{M+1}<q\le q_M,\ M\in\mathbb{N};\\ 0,\  \text{if}\ q>q_1.\end{cases}
$$
From this and \eqref{finalfixedq}, it follows that
$$
\begin{aligned}
&\lim_{n\to\infty}P_n^{\mathcal{M};q}(\mathcal{S}^\uparrow(n,M^*))=\begin{cases}\frac{q-1}q(1-q^{-M})\sum_{j=M}^\infty\frac1{q^j-1}\ge1-q^{-M},\\
\text{if}\
q_{M+1}<q\le q_M,\ M=1,2,\cdots;\\
\frac{q-1}q,\ \text{if}\ q>q_1.\end{cases}
\end{aligned}
$$
Now for any $M,L\in\mathbb{N}$ with $M<L$, the function $\frac{z-1}z(1-z^{-M})\sum_{j=M}^L\frac1{z^j-1}$
is analytic in $\{|z|>1\}$ and converges pointwise there to
$\frac{z-1}z(1-z^{-M})\sum_{j=M}^\infty\frac1{z^j-1}$. Thus, by the Weierstrass convergence theorem,
$\frac{z-1}z(1-z^{-M})\sum_{j=M}^\infty\frac1{z^j-1}$ is analytic in $\{|z|>1\}$.
In particular then, it follows that the function
$\frac{q-1}q(1-q^{-M})\sum_{j=M}^\infty\frac1{q^j-1}$ can
be equal to $\frac1e$ at no more than
 a finite number of $q\in(q_{M+1}, q_M]$.
 \hfill $\square$

\section{Proof of Theorem \ref{Suksec}}\label{pfSuksec}
\noindent \it Proof of part (i).\rm\  We apply \eqref{lucesecmin} with $\theta_i=n+1-i$. Let $M=M_n\sim bn$, with $b\in(0,1)$.
Noting that
$$
\sum_{i=1}^k(n+1-i)=\frac12n(n+1)-\frac12(n-k)(n-k+1),
$$
 we obtain
\begin{equation}\label{formulaMnLuce}
\begin{aligned}
&P_n^{\mathcal{L}_\text{inv};\{\theta_j\}_{j=1}^n}(\mathcal{S}^\downarrow(n,M_n))=\sum_{j=M_n+1}^n
\frac{2(n+1-j)}{n(n+1)}\frac{n(n+1)-(n-M_n)(n-M_n+1)}{n(n+1)-(n-j+1)(n-j+2)}\sim\\
&\frac2n\sum_{j=[bn]+1}^n\left(1-\frac j{n+1}\right)\frac{1-(1-b)^2}{1-(1-\frac{j-1}n)(1-\frac{j-1}{n+1})}\sim2
\int_b^1(1-x)\frac{1-(1-b)^2}{1-(1-x)^2}dx.
\end{aligned}
\end{equation}
From \eqref{formulaMnLuce} we have
\begin{equation}\label{Suklimit}
\begin{aligned}
&\lim_{n\to\infty}P_n^{\mathcal{L}_\text{inv};\{\theta_j\}_{j=1}^n}(\mathcal{S}^\downarrow(n,M_n))=
2\left(1-(1-b)^2\right)\int_b^1\frac{1-x}{1-(1-x)^2}dx,\ \ M_n\sim bn.
\end{aligned}
\end{equation}
Writing
$$
\frac{1-x}{1-(1-x)^2}=\frac{1-x}{2x-x^2}=\frac12\left(\frac1x-\frac1{2-x}\right),
$$
we obtain
$$
\int_b^1\frac{1-x}{1-(1-x)^2}dx=-\frac12\log b-\frac12\log(2-b)=-\frac12\log(2b-b^2).
$$
Thus, from \eqref{Suklimit},
\begin{equation}\label{finalformb}
\lim_{n\to\infty}P_n^{\mathcal{L}_\text{inv};\{\theta_j\}_{j=1}^n}(\mathcal{S}^\downarrow(n,M_n))=-(2b-2b^2)\log(2b-b^2).
\end{equation}
The function $-x\log x$, for $x\in(0,1)$, attains its maximum value of $\frac1e$  at $x=\frac1e$.
Thus, the right hand side  of \eqref{finalformb} as a function of $b\in(0,1)$ attains its maximum value of $\frac1e$ when
$2b-b^2=\frac1e$, or equivalently,
 $b=1-(1-\frac1e)^\frac12$. This proves part (i).
 \hfill $\square$

 \noindent \it Proof of part (ii).\rm\
 We apply \eqref{lucesecmin} with $\theta_i=i$. Let $M=M_n\sim bn$, with $b\in(0,1)$. We have
\begin{equation*}
\begin{aligned}
&P_n^{\mathcal{L}_\text{inv};\{\theta_j\}_{j=1}^n}(\mathcal{S}^\downarrow(n,M_n))=\sum_{j=M_n+1}^n
\frac j{\frac12n(n+1)}\frac{\frac12M_n(M_n+1)}{\frac12(j-1)j)}\sim
2b^2\sum_{j=[bn]+1}^n\frac1{j-1},
\end{aligned}
\end{equation*}
and thus,
\begin{equation}\label{formulaMnLuceagain}
\lim_{n\to\infty}P_n^{\mathcal{L}_\text{inv};\{\theta_j\}_{j=1}^n}(\mathcal{S}^\downarrow(n,M_n))=-2b^2\log b.
\end{equation}
Let $H(b)=-2b^2\log b$. Then $H'(b)=-4b\log b-2b$, and so
$H'$ vanishes when $\log b=-\frac12$; that is, $b=b^*=e^{-\frac12}$.
One has $H(b^*)=\frac1e$. This completes the proof of part (ii).
\hfill $\square$

\end{document}